\documentclass[11pt,a4paper]{amsart}
\usepackage{amssymb}
\usepackage{amsmath}
\usepackage{amsfonts}
\usepackage[english]{babel}
\usepackage[T1]{fontenc}
\usepackage[latin1]{inputenc}
\usepackage{amsthm}
\usepackage[all]{xy}
\usepackage{a4wide}

\date{17.02.2017}
\theoremstyle{definition}
\newtheorem{thm}{Theorem}[section]

\newtheorem{cor}[thm]{Corollary}
\newtheorem{prop}[thm]{Proposition}

\theoremstyle{plain}
\newtheorem{theorem}{Theorem}[section]

\theoremstyle{remark}

\newtheorem{example}[theorem]{Example}

\begin{document}

\begin{abstract}
We study some basic properties of the class of universal operators on Hilbert space,
and provide new examples of universal operators and universal pairs.
\end{abstract}

\thanks{First author is supported by the Magnus Ehrnrooth Foundation in Finland.}

\title[On universal operators and universal pairs]{On universal operators and universal pairs}

\author[R. Schroderus]{Riikka Schroderus}
\address{Department of Mathematics and Statistics\\ University of Helsinki, Box 68 \\ FI-00014 Helsinki, Finland}
\email{riikka.schroderus@helsinki.fi}

\author[H.-O. Tylli]{Hans-Olav Tylli}
\address{Department of Mathematics and Statistics\\ University of Helsinki, Box 68 \\ FI-00014 Helsinki, Finland}
\email{hans-olav.tylli@helsinki.fi}

\maketitle

\section{Introduction}

Let $\mathcal{L}(H)$ be the algebra of bounded linear operators on the separable infinite-dimensional Hilbert space $H$. Recall that the operators $T_1\in \mathcal{L}(H_1) $ and $T_2\in \mathcal{L}(H_2) $ are similar if there exists a linear isomorphism $J:H_1\longrightarrow H_2 $ such that $ T_1 =J^{-1}T_2 J$. The operator $U\in \mathcal{L}(H)$ is called \textit{universal} if for any $T\in \mathcal{L}(H)$ there exist a closed $U$-invariant subspace $M\subset H$, i.e. $U(M)\subset M$, and a constant $c\neq 0 $ such that the operators $ U_{|M}: M \longrightarrow M$ and $ c T : H \longrightarrow H$ are similar.

The concept of a universal operator was introduced by Rota \cite{Ro}, where he showed that the backward shift of infinite multiplicity is an explicit example of these seemingly strange objects. The invariant subspace problem provides motivation for studying concrete universal operators, namely, every operator in $\mathcal{L}(H)$ has a non-trivial invariant subspace if and only if the minimal non-trivial invariant subspaces of (any) universal operator are $1$-dimensional. General references for this approach and results about universal operators are \cite[Chapter 8]{CP} and the survey \cite{CG1}.

Later Caradus \cite{C} (see also \cite[Theorem 8.1.3]{CP}) exhibited a simple sufficient condition for an operator to be universal, namely, 

\begin{itemize}
\item[(C)] \textit{If $U \in \mathcal{L}(H)$ is such that the kernel $\textnormal{Ker}\, (U)$
is infinite-dimensional and its range $\textnormal{Ran}\,(U) = H$, then $U$ is universal for $H$.
}
\end{itemize}

\noindent However, Caradus' condition and its recent generalisation by Pozzi \cite{Po} 
are very far from being necessary.
In Section \ref{class} of  this paper we look more closely  at some fundamental  properties of  the class of universal operators, as well as some of its subclasses. In particular, we derive spectral theoretic consequences of universality which can be used to verify that a given operator is not universal. 

Because of its simplicity, condition (C) is often used to obtain examples of universal operators, even though the desired properties can be difficult to verify for many concrete operators.
By a celebrated example of Nordgren, Rosenthal and Wintrobe \cite{NRW} from 1987 
the operators $C_{\varphi} -\lambda I$ are universal on the Hardy space $H^2 (\mathbb{D})$, 
whenever the composition operator $f \longmapsto C_{\varphi}f = f\circ \varphi$ is induced by a hyperbolic automorphism $\varphi$ of the unit disc $\mathbb{D}$ and 
$\lambda$ belongs to the interior of the spectrum of  $C_{\varphi}$.
In this case the infinite-dimensionality of $\textnormal{Ker}\, (C_{\varphi} -\lambda I)$
is verified by explicit computation, but the original argument for the surjectivity 
relies on fairly sophisticated properties of multiplication operators induced by 
certain Blaschke products in $H^2 (\mathbb{D})$. 
Only very recently an alternative argument for the universality of
$C_{\varphi} -\lambda I$ on $H^2 (\mathbb{D})$ was given in  \cite{CG2}.
For other concrete examples of universal operators, see e.g. \cite{PP, Po, CG0, CG2a}.
Moreover, the connection between the invariant subspace problem and universality has motivated recent work on the lattice of invariant subspaces of $C_{\varphi}$ on $H^2 (\mathbb{D})$ for
hyperbolic automorphisms $\varphi$, see e.g. \cite{Ma1}  and \cite{GG}. 

In Section \ref{secunivex} we show that the adjoint $C_{\varphi}^{*} -\overline{\lambda} I$ is universal on $S^2 (\mathbb{D})$, the Hilbert space consisting of analytic functions $f: \mathbb{D} \longrightarrow\mathbb{C}$ such that $f'\in H^2 (\mathbb{D})$, whenever $\lambda$ is an interior point of the spectrum of $C_{\varphi}$ on $S^2 (\mathbb{D})$.
It follows from known results that $C_{\varphi} -\lambda I$ is not a universal operator on $S^2 (\mathbb{D})$, for any $\lambda \in \mathbb{C}$, which suggests that universality
passes to the adjoint for small enough spaces in the scale of weighted Dirichlet spaces of analytic functions on $\mathbb{D}$. 

Recently M\"uller \cite{Mu2} introduced a notion of universality for commuting $n$-tuples 
of operators, and he obtained versions of the sufficient condition (C) in this setting. 
However, examples of universal commuting  $n$-tuples are rather more difficult to exhibit 
compared to the case of a single operator, and in Section \ref{pairs} we discuss new
 concrete examples of universal commuting  pairs $(U_1, U_2)\in \mathcal{L}(H)^2$.
In particular, we show that certain pairs $(L_A,R_B)$ of left and right multiplication operators 
on the ideal of the Hilbert-Schmidt operators are universal and consider the case of 
universal NRW-pairs $(C_{\varphi} -\lambda I, C_{\psi} -\mu I)$ in 
$\mathcal{L}\big(H^2 (\mathbb{D})\big)^2$.

\section{Structure of the class of universal operators}\label{class}

The main interest has been in exhibiting and analysing concrete examples of universal operators 
belonging to various classes of operators, and less attention has been paid to general properties 
of the full class 
\[
\mathcal U(H) = \{U \in \mathcal{L}(H): U\  \textrm{is universal}\}.
\]
In this section we  systematically consider 
$\mathcal U(H)$ and some of its subclasses. Clearly $\mathcal U(H_1)$ and 
$\mathcal U(H_2)$ are related by similarity whenever $H_1$ and $H_2$
are separable infinite-dimensional Hilbert spaces, so the particular
realisation of the Hilbert space $H$ is immaterial.  
We will use the notation 
 $B_\infty: (\oplus_{\mathbb{Z}_{+}} \ell^2)_{\ell^2} \to (\oplus_{\mathbb{Z}_{+}} \ell^2)_{\ell^2}$ for 
Rota's universal model operator, 
\[
B_\infty x = B_\infty(x_0,x_1,\ldots) = (x_1, x_2, \ldots), 
\]
for $x = (x_n)_{n\in \mathbb{Z}_{+}} \in (\oplus_{\mathbb{Z}_{+}} \ell^2)_{\ell^2}$, 
where $x_n \in \ell^2$ for any $n \ge 0$.
The universality of the backward shift $B_\infty$ of infinite multiplicity 
on  $(\oplus_{\mathbb{Z}_{+}} \ell^2)_{\ell^2}$ is immediate from  (C),
but the original argument by Rota \cite{Ro}  is quite direct.

It was pointed out in \cite[p. 44]{CG1} that the precise relationship between universality and 
condition (C) is somewhat circular:  $U \in \mathcal U(H)$ if and only if there is a $U$-invariant 
infinite-dimensional subspace $M \subset H$ so that the restricted operator
$U_{|M}: M \to M$ satisfies condition (C). 
This is seen by recalling that the restriction of 
any  $U \in \mathcal U(H)$ to some invariant subspace is similar 
to $cB_\infty$ for some $c \neq 0$, 
combined with an observation of Pozzi recalled separately as Proposition \ref{Poz} below. 
To state the proposition in a convenient form we write
operators $V \in \mathcal{L}(H)$ with respect to direct sum decompositions 
$H = M \oplus M^{\perp}$
as vector-valued operator matrices
\[
V = \left( \begin{array}{ccc}
V_{1,1} & V_{1,2}\\
V_{2,1} & V_{2,2}\\
\end{array} \right)
\]
in the obvious fashion. Thus $V_{1,1} = PV_{|M}$, $V_{1,2} = PV_{|M^{\perp}}$,
$V_{2,1} = QV_{|M}$ and $V_{2,2} = QV_{|M^{\perp}}$,
where $P: H \to M$ and $Q = I - P: H \to M^{\perp}$ are the related orthogonal projections.
The following fact allows to construct examples of universal operators 
having different properties on direct sums $H \oplus H$.

\begin{prop}\label{Poz} (\cite[Remark 1.4]{Po}) 
Let $H = M \oplus M^{\perp}$, where $M \subset H$ is an infinite-dimensional subspace,
and let
\[
V = \left( \begin{array}{ccc}
U & A\\
0 & B\\
\end{array} \right) \in \mathcal{L}(H)
\]
as above.
Suppose that $U \in \mathcal{L}(M)$ is a universal operator for $M$. 
Then $V \in \mathcal U(H)$ for any operators $A$ and $B$.
\end{prop}

\begin{proof}
If $T \in \mathcal{L}(H)$ is given there is, by assumption, a
$U$-invariant subspace $N \subset M$,
and $c \neq 0$ such that $U_{|N}: N\to N $ and $ cT : H\to H$ are similar. Fix an isometry $J_0: M \to H$. We have that $U_{|N}$ is similar to $c J_0^{-1}TJ_0$. 
Since $N \subset M$ we get that $N$ is  $V$-invariant,
$V_{|N} = U_{|N}$ is similar to $c J_0^{-1}TJ_0$
and consequently also to $cT$.
\end{proof}

We are interested in conditions that enable us to decide 
whether a given concrete operator is universal or not, and 
we first consider spectral criteria. Let $\sigma(S; H)$ denote the spectrum of 
$S \in \mathcal{L}(H)$. 
The spectrum of a diagonal sum of operators on $H_1 \oplus H_2$ satisfies 
\[
\sigma\Big( \left( \begin{array}{ccc}
U & 0\\
0 & B\\
\end{array} \right); H_1 \oplus H_2\Big)
= \sigma(U; H_1) \cup \sigma(B; H_2)
\]
for any  $U \in \mathcal{L}(H_1)$ and $B \in \mathcal{L}(H_2)$.
It follows from Proposition \ref{Poz} that there is no general characterisation of universal operators purely in terms of their spectra.
Nevertheless, the universality of $U \in \mathcal U(H)$ does have relevant consequences 
for various subsets of the spectrum of $U$.
For this recall the classes of semi-Fredholm operators
\begin{align*}
\Phi_-(H) & = \{S \in \mathcal{L}(H): \dim \, (H/\textnormal{Ran}\, (S)) < \infty\},\\ 
\Phi_+(H) & = \{S \in \mathcal{L}(H): \dim \, (\textnormal{Ker}\,(S)) < \infty, \textnormal{Ran}\,(S)\  \textrm{is closed}\},
\end{align*} 
where $\Phi(H) = \Phi_+(H) \cap \Phi_-(H)$ consists of the Fredholm operators.
Operators $S \in \Phi_+(H)$ cannot be universal, since clearly $\textnormal{Ker}\,(S)$
has to be infinite-dimensional for $S$ to be an universal operator.
We will need the $\Phi_+$-spectrum of $S \in \mathcal{L}(H)$ defined as
\[
\sigma_e^+(S; H) = \{\lambda \in \mathbb C:  S -\lambda I\notin \Phi_+(H)\}.
\]
It is known \cite[Chapter III.19]{Mu} that  $\sigma_e^+(S; H)$ is a 
non-empty compact subset of  the essential spectrum
\[
\sigma_e(S; H) = \{\lambda \in \mathbb C:  S -\lambda I \notin \Phi(H)\}
\]
of $S$. Furthermore, let $\sigma_p(S; H)$ denote the point spectrum of $S$.

It follows from the definition of universality that Riesz operators $S \in \mathcal{L}(H)$
can not be universal. (Recall that $S$ is a Riesz  operator if $\sigma_e(S; H) = \{0\}$.)  
The following result reveals some further common spectral properties of universal operators.

\begin{thm}\label{spec}
Let $U \in \mathcal U(H)$ be an arbitrary universal operator. 
Then the following hold:
\begin{itemize}
\item[(i)]  There is $r > 0$ such that  the open disk 
\[
B(0,r) \subset \sigma_p(U; H) \cap \sigma_e^+(U; H) \subset \sigma_e(U; H) 
\subset \sigma(U; H),
\]
and, moreover, any $\lambda \in B(0,r)$ is an eigenvalue of $U$ having infinite multiplicity.
In particular, if  $U \in \mathcal U(H)$ then $0$ is an interior point
of any of the sets $\sigma_e^+(U; H)$, $\sigma_e(U; H)$, $\sigma_p(U; H)$ as well as  
$\sigma(U; H)$. 

\item[(ii)] There is $r > 0$ and a vector-valued holomorphic map
$z \mapsto y_z: B(0,r) \to H$ for which
\[
Uy_z = z y_z, \quad z  \in B(0,r).
\]

\end{itemize}
\end{thm}

\begin{proof}
We first recall some well-known spectral properties of the backward shift
$B_\infty$ on the direct sum $H_0 \equiv (\oplus_N \ell^2)_{\ell^2}$. 
Let $\vert z \vert< 1$ and fix the non-zero vector $x_0 \in \ell^2$, whence 
the sequence $x_z = (z^nx_0)_{n\ge 0} = (x_0,zx_0, \ldots) \in H_0$. Clearly
\[
B_\infty(x_z) = (zx_0,z^2x_0, \ldots) = zx_z,
\]
so that any $\vert z \vert< 1$ is an eigenvalue of infinite multiplicity for $B_\infty$,
that is $z \in \sigma_e^+(B_\infty; H_0)$. 
Moreover, $z \mapsto x_z$ is a (weakly) holomorphic map
$B(0,1) \to H_0$, since 
\[
z  \mapsto \langle x_z,y \rangle = \sum_{n=0}^\infty z^n \langle x_0,y_n \rangle
\]
is analytic for any $y = (y_n) \in H_0$, where $\langle \cdot,\cdot \rangle$ 
denotes the respective inner-product.

Let $U$ be an arbitrary universal operator on $H$.
By assumption  there is a constant $c \neq 0$, a $U$-invariant subspace 
$M \subset H$ and a linear isomorphism $J:  H_0 \to M$ so that
 $U_{|M} = J(cB_\infty)J^{-1}$. 
Since eigenvalues are preserved under similarity we get that
\[
B(0,\vert c\vert) \subset  \sigma_p(cB_\infty; H_0) =
\sigma_p(U_{|M}; M) \subset \sigma(U; H).
\]
Towards the related claim for  $\sigma_e^+(U; H)$ one obtains instead that
\[
B(0,\vert c\vert) \subset \sigma_e^+(cB_\infty; H_0) =
\sigma_e^+(U_{|M}; M)  \subset \sigma_e^+(U; H).
\]
For the right-hand inclusion note e.g. that  
$Ker(\lambda I_M - U_{|M}) \subset Ker(\lambda I - U)$, 
where the left-hand  subspace is infinite-dimensional by similarity, since  
 $\dim \, (Ker(\lambda I - cB_\infty)) = \infty$. 

Finally, the above identities $B_\infty x_z = z  x_z$
and $cJB_\infty = (U_{|M})J$ imply that 
\[
U(Jx_z) = cJB_\infty (x_z) = c z J(x_z), \quad \vert z \vert < 1.
\]
It follows that the renormalised holomorphic map $z \mapsto y_z \equiv J(x_{z/c})$ 
satisfies condition (ii) in the disk $B(0, \vert c \vert)$.
\end{proof}

We next state some typical applications of the preceding result.

\begin{cor}\label{int}
The operator $T \in \mathcal{L}(H)$ can not be universal if any  
of the following conditions holds:
\begin{itemize}
\item[(i)] the interior $int \,(\sigma_p(T; H)) = \emptyset$, 
\item[(ii)]the interior $int \,(\sigma_e(T; H)) = \emptyset$,
\item[(iii)] every non-zero eigenvalue $\alpha \in \sigma_p(T; H)$ has finite multiplicity.
\end{itemize}
\end{cor}

Another immediate consequence of Theorem \ref{spec} which will be useful in Section \ref{secunivex} reads as follows.

\begin{cor}\label{bdry} Suppose that $T\in \mathcal{L} (H)$ and $\lambda \in \partial \sigma (T; H)$. Then $T-\lambda I$ can not be universal. In particular, if $ \sigma (T; H) = \partial \sigma (T; H)$, then $T-\lambda I$ is not universal for any $\lambda \in \mathbb{C}$.
\end{cor}

We next consider general properties of the class $\mathcal U(H)$  of the universal operators.
Recently Pozzi \cite[Thm. 3.8]{Po} extended Caradus' condition (C)  
as follows:
\begin{itemize}
\item[(C$_+$)] \textit{If $U \in \mathcal{L}(H)$ satisfies $\dim \, Ker(U) = \infty$
and $\dim \,(H/\textnormal{Ran}\,(U))  < \infty$, 
then $U \in \mathcal U(H)$.}
\end{itemize}
It is helpful for comparative purposes to introduce the subclasses 
\begin{align*}
\mathcal UC(H) & = \{U \in \mathcal{L}(H): U\  \textrm{satisfies}\  (C)\},\\
\mathcal UC_+(H) & = \{U \in \mathcal{L}(H): U\  \textrm{satisfies}\ (C_+)\}
\end{align*}
of $\mathcal U(H)$. 
Observe that $\mathcal UC_+(H) = \Phi_-(H) \setminus \Phi_+(H)$.
Hence it follows from the classical perturbation theory for semi-Fredholm operators that
$\mathcal UC_+(H)$
is preserved under sufficiently small operator norm perturbations 
as well as compact perturbations, see \cite[Thm. 4.2]{CPY} or 
\cite[Thms. III.16.18 and III.16.19]{Mu}.
In particular,
 $U + K \in \mathcal UC_+(H)$ whenever 
$U \in \mathcal UC(H)$ and $K$ is a compact operator. 
In the sequel, we denote by $ \mathcal{K}(H)$ the closed ideal of $ \mathcal{L}(H)$ 
consisting of compact operators on $H$.
Moreover, in \cite[Thm. 2]{CG0} the authors obtained by direct means 
a perturbation result for the class $\mathcal UC(H)$
which contains more detailed information. 
We also recall that the universal model operator $B_\infty$ has the stronger property that
its restrictions represent suitable multiples $cT$ up to unitary equivalence for any 
$T \in \mathcal{L}(H)$, see e.g. \cite[Thm. 8.1.5]{CM} or \cite[Chapt. 1.5]{RR}.

It is evident from Proposition \ref{Poz} that the subclasses  $\mathcal UC(H)$ 
and $\mathcal UC_+(H)$ 
are much smaller than $\mathcal U(H)$, and $\mathcal U(H)$ 
contains operators very different from $B_\infty$. Moreover, 
$\mathcal{U}C(H)$ is not preserved by compact perturbations.
For the record we include related very simple examples.

\begin{example}\label{notinc} 
(i) Let $U\in \mathcal{U}(H)$, so that 
\[
V = \left( \begin{array}{ccc}
U & 0\\
0 & B\\
\end{array} \right) 
\] 
is a universal operator on $H\oplus H$ for any $B \in \mathcal{L}(H)$
by Proposition \ref{Poz}. For instance, 
if  $B(H)$ has infinite codimension in $H$, then
 $\textnormal{Ran}(V) = U(H)+ B(H)$ has infinite codimension in $H\oplus H$, and if 
 $\textnormal{Ran} \, (B)$ is not closed, then $\textnormal{Ran} \, (V)$ is not closed either.

(ii) Define $U \in \mathcal L(\ell^2)$ by $Ue_{2n}=e_n$ and $Ue_{2n+1}=0$ for $n\in \mathbb{N}$,
so that $U\in \mathcal UC(\ell^2)$. Let $K \in \mathcal K(\ell^2)$ be the rank-$1$ operator defined by $Ke_2 = -e_1$ and $Ke_n=0$ for $n\neq 2$. Since $(U+K)e_2 = 0$ it follows that 
$e_1\notin (U+K)(\ell^2)$, so that $U+K \notin \mathcal UC(\ell^2)$ 
(even though $U+K \in \mathcal UC_+(\ell^2)$).

\end{example}

Explicit examples demonstrate similarly that the full class
$\mathcal U(H)$ of universal operators is neither  open in the operator norm 
nor invariant under  compact perturbations.

\begin{example}\label{pert}
 Fix a universal operator $U \in \mathcal U(H)$.  Proposition \ref{Poz} implies that
\[
V = \left( \begin{array}{ccc}
U & I_H\\
0 & 0\\
\end{array} \right)
\]
is a universal operator on $H \oplus H$, where $I_H$
is  the identity map of $H$. Consider the sequence $(V_n) \subset \mathcal L(H \oplus H)$ defined by
\[
V_n = \left( \begin{array}{ccc}
U & I_H\\
\frac{1}{n}I_H & 0\\
\end{array} \right),
\]
that is, $V_n(x,y) = (Ux+y,\frac{1}{n}x)$ for $(x,y) \in H \oplus H$.
Note that $V_n$ is not universal on $H \oplus H$ for any
$n \in \mathbb{N}$, since $Ker(V_n) = \{(0,0)\}$. In fact, 
$V_n(x,y) = (Ux+y,\frac{1}{n}x) = (0,0)$ yields that $x = 0$ and 
$y = -Ux = 0$. Clearly 
$\Vert V_n - V\Vert = \frac{1}{n}$ for $n \in \mathbb N$,
so that   $V \in \mathcal U(H \oplus H)$ is not an interior point of
$\mathcal U(H \oplus H)$.

Furthermore, let   $K \in \mathcal K(H)$ be the diagonal operator defined by 
$Kf_n = \frac{1}{n}f_n$ for $n \in \mathbb{N}$,
where $(f_n)$ is some fixed orthonormal basis of $H$. 
Consider 
\[
W = \left( \begin{array}{ccc}
U & I_H\\
K & 0\\
\end{array} \right) 
=  V +
\left( \begin{array}{ccc}
0 & 0\\
K & 0\\
\end{array} \right) \in \mathcal L(H \oplus H),
\]
that is, $W(x,y) = (Ux + y,Kx)$ for $(x,y) \in H \oplus H$.
Thus $W$ is a compact perturbation of  $V \in \mathcal U(H \oplus H)$, but
$W$ is not a universal operator, since again 
$Ker(W) = \{(0,0)\}$.
\end{example}

It follows from the algebraic semi-group property of $\Phi_-(H)$, 
see \cite[Thm. III.16.5]{Mu}, that the subclass
$\mathcal UC_+(H)$ is multiplicative in the sense that 
$UV \in \mathcal UC_+(H)$ whenever $U, V \in \mathcal UC_+(H)$ (and this property 
is obvious 
for $\mathcal UC(H)$). Multiplicativity easily fails for the class $\mathcal U(H)$.
In fact, fix $U_0, V_0 \in \mathcal U(H)$.
Then 
$U = \left( \begin{array}{ccc}
U_0 & 0\\
0 & 0\\
\end{array} \right)$ and 
$V =  \left( \begin{array}{ccc}
0 & 0\\
0 & V_0\\
\end{array} \right)$ belong to $\mathcal U(H \oplus H)$
by Proposition \ref{Poz}, but $UV = 0$.

\section{Universality of the adjoint $C_{\varphi}^{*}- \overline{\lambda} I$ on $S^2 (\mathbb{D})$}\label{secunivex}

Recall that Nordgren, Rosenthal and Wintrobe \cite{NRW} showed that 
the operators $C_{\varphi} -\lambda I$ are universal on the Hardy space 
$H^2 (\mathbb{D})$ for any hyperbolic automorphism $\varphi$ of the unit disc  $\mathbb{D}$ 
and $\lambda \in int \,(\sigma(C_{\varphi}; H^2 (\mathbb{D})))$. Here 
$C_{\varphi}$ is the composition operator $f \mapsto f\circ \varphi$.
In this section we will discuss potential analogues of this result in the scale of weighted Dirichlet spaces 
$D_\beta(\mathbb{D})$, which are Hilbert spaces of analytic functions 
defined on the unit disc  $\mathbb{D}$.
Our main observation (Theorem \ref{thmadj})
is that  the \textit{adjoint} $C_{\varphi}^{*} - \overline{\lambda} I$
is universal on the space $S^2 (\mathbb{D})$, 
whenever $\varphi$ is a hyperbolic automorphism of 
 $\mathbb{D}$ and $\lambda \in int \,(\sigma \big( C_{\varphi}; S^2(\mathbb{D})))$. 
Here $S^2(\mathbb{D})$ is the Hilbert space consisting of the analytic functions 
$f: \mathbb{D} \longrightarrow\mathbb{C}$ such that $f'\in H^2 (\mathbb{D})$, whence 
$S^2(\mathbb{D})$ is continuously embedded in the classical Dirichlet space 
$\mathcal{D}^2$. 

Recall for $\beta \in \mathbb{R}$ that the weighted Dirichlet space 
$\mathcal{D}_{\beta}(\mathbb{D})$ is the Hilbert space of analytic functions 
$f(z)=\sum_{n=0}^{\infty} a_n z^n$ that satisfy 
\[
\Vert f\Vert^2_{\beta} = \sum_{n=0}^{\infty}  \vert a_n \vert^2 (n+1)^{2\beta} < \infty.
\] 
(These spaces are also special cases of the weighted Hardy spaces.)
The Hardy space $H^2 (\mathbb{D})$ is obtained for $\beta = 0$, 
the Bergman space
$A^2 (\mathbb{D})$ for $\beta = - 1/2$, the Dirichlet space $\mathcal{D}^2$ for $\beta = 1/2$
and $S^2 (\mathbb{D})$ for $\beta = 1$ (possibly up to an equivalent norm). 
We also recall that there is a continuous embedding 
$D_\beta(\mathbb{D}) \subset D_\alpha(\mathbb{D})$  for $\alpha < \beta$.
The reference \cite[Chapter 2.1]{CM} contains more background
about these spaces.

It will be enough for our purposes to consider the 
normalized hyperbolic automorphisms of $\mathbb{D}$ that have the form
\begin{equation}\label{varphir}
\varphi_r (z)= \frac{z+r}{1+rz}, \quad r\in (0,1).
\end{equation} 
In fact, it is known that all other hyperbolic automorphisms of  $\mathbb{D}$ 
can be conjugated by automorphisms of $\mathbb{D}$ to the preceding normalised form. 
We will later need the fact that 
$$
\varphi_r^{-1} (z)= \varphi_{-r} (z)= \frac{z-r}{1-rz} ,
$$
belongs to the same conjugacy class as $\varphi_r$,  since 
$\varphi_{-r} = g \circ \varphi_r \circ g $, where 
$g(z) =-z $ for $z \in \mathbb{D}$. Hence 
$ 
C_{\varphi_{-r}} = C_{g}C_{\varphi_r}C_{g},
$
so that $C_{\varphi_r} $ and $C_{\varphi_{-r}}= C_{\varphi_r}^{-1}$ are similar operators.
For more information on linear fractional transformations in general, see e.g. \cite[Chapter 0]{Sh}, and on composition operators acting on spaces of analytic functions, see \cite{CM}
or  \cite{Sh}. 

The composition operators 
$C_{\varphi_{r}}$ are known to be bounded on $\mathcal{D}_{\beta}(\mathbb{D})$ for all 
$\beta \in \mathbb R$, see \cite[Chapter 3.1]{CM} and \cite{Z} for various ranges of $\beta$.
We will require the result that 
\begin{equation}\label{spectrum}
\sigma \big( C_{\varphi_r} ; S^2 (\mathbb{D})\big) = \sigma \big( C_{\varphi_r} ; H^2 (\mathbb{D}))
= 
\Big\{ \lambda \in \mathbb{C} : \Big(\frac{1-r}{1+r} \Big)^{1/2} \leq \vert\lambda\vert \leq \Big(\frac{1-r}{1+r} \Big)^{-1/2}\Big\}
\end{equation}
for all $0 < r < 1$. 
We refer to \cite[Thm. 7.4]{CM} for the Hardy space case and to \cite[Thm. 3.9]{GS} 
for the case $S^2(\mathbb{D}) = \mathcal{D}_1(\mathbb{D})$.
In the sequel we will denote the corresponding open annulus, 
i.e. the interior of the above spectrum, by
$$ 
\mathcal{A}_r  := \Big\{ \lambda \in \mathbb{C} : \Big(\frac{1-r}{1+r} \Big)^{1/2} < \vert\lambda\vert < \Big(\frac{1-r}{1+r} \Big)^{-1/2}\Big\}.
$$

We point out as an  initial motivation that 
$C_{\varphi_r} - \lambda I$ is not universal on 
any of the small weighted Dirichlet spaces contained in 
the classical Dirichlet space $\mathcal{D}^2 = \mathcal{D}_{1/2}(\mathbb{D})$.

\begin{example}\label{nouniv}
Let $0 < r < 1$. Then $C_{\varphi_r} - \lambda I$ is not universal on $\mathcal{D}_{\beta}(\mathbb{D})$ for any $\beta \ge 1/2$ and any $\lambda \in \mathbb{C}$.

In fact, for $\beta = 1/2$ it is known that $\sigma \big( C_{\varphi_{r}} ; \mathcal{D}^2\big) =\mathbb{T}$ by \cite[Thm.  3.2]{Hi}. Hence it follows from Corollary \ref{bdry} that neither 
$C_{\varphi_{r}} -\lambda I$ nor its adjoint $C_{\varphi_r}^{*} -\overline{\lambda} I$ 
is  universal on $\mathcal{D}^2$ for any 
$\lambda \in \mathbb{C}$. 

For $\beta > 1/2$ it follows from \cite[Thm. 3.9]{GS} and its proof that in this case
the point spectrum $\sigma_p \big(C_{\varphi_r};S^2 (\mathbb{D})\big) =\{1\}$, 
and moreover that 
$ \textup{Ker}\, \big(C_{\varphi_r} - I ; S^2 (\mathbb{D}) \big) =\mathbb{C}$.
Consequently  $\dim \textup{Ker}\, \big(C_{\varphi_r} -\lambda I ; S^2 (\mathbb{D})\big)$ 
is either $0$ (for $\lambda \neq 1$) or $1$ (for $\lambda = 1$), so 
Corollary \ref{int} yields that 
$C_{\varphi_{r}} -\lambda I$ can not be universal on the weighted Dirichlet spaces for 
any $\beta > 1/2$ and $\lambda \in \mathbb{C}$.
\end{example}

As a contrast we show in the main result of this section 
that the adjoint of $C_{\varphi_{r}} -\lambda I$ is universal  on $S^2 (\mathbb{D})$. 

\begin{thm}\label{thmadj}
Let $ \varphi_r$ be the hyperbolic automorphism of $\mathbb{D}$ defined by 
$\varphi_r (z)= \frac{z+r}{1+rz}$ for $r\in (0,1)$. Then 
$C_{\varphi_r}^{*} -\overline{\lambda} I $ is universal on $S^2 (\mathbb{D})$ 
for any $\lambda \in \mathcal{A}_r$. 
\end{thm}

\begin{proof} 
Let $0 < r < 1$ and write $H^2 (\mathbb{D}) = zH^2 (\mathbb{D}) \oplus [\mathbf{1}]$, where
$[\mathbf{1}]$ denotes the constant functions.
The crux of the argument is the fact that the compression 
$$ 
P_{z H^2}C_{\varphi_r}^{-1} : z H^2 (\mathbb{D}) \longrightarrow z H^2 (\mathbb{D})
$$
and the restriction of the adjoint 
$$ 
C_{\varphi_r}^{*} :  z S^2 (\mathbb{D}) \longrightarrow z S^2 (\mathbb{D})
$$
are similar operators, where the subspace $z S^2 (\mathbb{D}))$ is invariant 
under $C_{\varphi_r}^{*}$.
The details of the similarity are explained in Corollary 3.6 and Remark 3.8 in \cite{GS},
 which in turn is based on a duality argument of Hurst  \cite[Thm. 5]{Hu}.
  
We first consider the operator $C_{\varphi_r}^{-1} -\lambda I: H^2 (\mathbb{D}) \longrightarrow H^2 (\mathbb{D})$. 
Since  $C_{\varphi_r}^{-1}(\mathbf{1})=\mathbf{1}$ we may write 
$C_{\varphi_r}^{-1} -\lambda I $ as the following operator matrix acting on 
$H^2 (\mathbb{D}) = z  H^2 (\mathbb{D})\oplus [\mathbf{1}]$:
$$
C_{\varphi_r}^{-1} -\lambda I = \left( \begin{array}{cc}
P_{z H^2}C_{\varphi_r}^{-1} - \lambda I & 0\\
P_{\mathbb{C}}C_{\varphi_r}^{-1}  & (1-\lambda )
\end{array}
\right).
$$

We claim that the compression ${P_{z H^2}C_{\varphi_r}^{-1}} -\lambda I: z H^2 (\mathbb{D}) \longrightarrow z H^2 (\mathbb{D})$ satisfies Caradus' condition (C)
for all $\lambda\in \mathcal{A}_r$. 
Since $C_{\varphi_r}^{-1} = C_{\varphi_{-r}}$ we know that the operator
$ C_{\varphi_r}^{-1} -\lambda I:  H^2 (\mathbb{D}) \longrightarrow  H^2 (\mathbb{D}) $ 
is surjective by the proof of \cite[Thm. 6.2]{NRW}.
Hence it follows that 
$(P_{z H^2}C_{\varphi_r}^{-1} - \lambda I)(z H^2 (\mathbb{D})) = 
z  H^2 (\mathbb{D})$ as well.
In fact, if  $g \in zH^{2}(\mathbb{D})$ is arbitrary, then there is  
$f =f_1 + f_2 \in H^2 (\mathbb{D})$, with $f_1\in z  H^2 (\mathbb{D})$ and $f_2 \in [\mathbf{1}]$, such that $( C_{\varphi_r}^{-1} -\lambda I) f=g$, 
that is
$$ 
\left( \begin{array}{cc}
P_{z H^2}C_{\varphi_r}^{-1} -\lambda I& 0\\
P_{\mathbb{C}}C_{\varphi_r}^{-1} & (1-\lambda )
\end{array}
\right) 
 \left( \begin{array}{c}
f_1\\
f_2
\end{array}
\right)
=
\left( \begin{array}{c}
( P_{z H^2}C_{\varphi_r}^{-1} -\lambda I)f_1 \\
P_{\mathbb{C}}C_{\varphi_r}^{-1} f_1  + (1-\lambda) f_2
\end{array}
\right) 
=
 \left( \begin{array}{c}
g\\
0
\end{array}
\right) .
$$
In particular, $(P_{z H^2}C_{\varphi_r}^{-1} -\lambda I)f_1 = g$, 
so that the compression 
${P_{z H^2}C_{\varphi_r}^{-1}} - \lambda I$  is an onto map
$ z  H^2 (\mathbb{D}) \longrightarrow z  H^2 (\mathbb{D})$ for $\lambda\in \mathcal{A}_r$.
Moreover, it is not difficult to check that since $\lambda \in \mathcal{A}_r$
 is an eigenvalue of infinite multiplicity for $ C_{\varphi_r}^{-1} : H^2 (\mathbb{D}) \longrightarrow H^2 (\mathbb{D}) $, the same fact holds for the compression
$P_{z H^2}C_{\varphi_r}^{-1} : z H^2 (\mathbb{D}) \longrightarrow z H^2 (\mathbb{D})$. 
 Consequently $ P_{z H^2}C_{\varphi_r}^{-1} -\lambda I: z H^2 (\mathbb{D}) \longrightarrow z H^2 (\mathbb{D})$ satisfies (C).

It follows from the similarity stated in the beginning of the argument that 
the restricted adjoint $C_{\varphi_r}^{*} -\lambda I:  z S^2 (\mathbb{D}) \longrightarrow z S^2 (\mathbb{D}) $ 
also satisfies (C) and is hence universal on $z S^2 (\mathbb{D})$. 
Write $C_{\varphi_r}^{*} -\lambda I $ on $S^2 (\mathbb{D})$ as an operator matrix acting on $ z S^2 (\mathbb{D})\oplus [\mathbf{1}]$, that is,
$$\left( \begin{array}{cc}
C_{\varphi_r}^{*} -\lambda I & P_{z S^2}C_{\varphi_r}^{*}\\
0 & (1-\lambda )
\end{array}
\right),  
$$
where we also take into account that $C_{\varphi_r}^{*}(z S^2 (\mathbb{D})) \subset z S^2 (\mathbb{D})$. 
It follows that $C_{\varphi_r}^{*} -\lambda I:   S^2 (\mathbb{D}) \longrightarrow  S^2 (\mathbb{D}) $ is universal by Proposition \ref{Poz}.

Alternatively, in the last step one may also note that if $\lambda \neq 1$, then $C_{\varphi_r}^{*} -\lambda I:   S^2 (\mathbb{D}) \longrightarrow  S^2 (\mathbb{D}) $ satisfies (C), while 
$\textup{codim} \, \textup{Ran} \, ( C_{\varphi_r}^{*} - I )=1$ in $S^2 (\mathbb{D})$,
so that $C_{\varphi_r}^{*} - I $ satisfies the generalised condition (C$_+$).

Finally, note that the annulus $\mathcal{A}_r$ is preserved by complex conjugation, so that
we may above change $C_{\varphi_r}^{*} -\lambda I$ to $C_{\varphi_r}^{*} -\overline{\lambda} I$.
\end{proof}

Heller \cite{Heller} found a concrete formula for the adjoint 
$C_{\varphi_r}^{*}$ on $S^2 (\mathbb{D})$ which involves a compact perturbation. 
This fact leads to a related universal operator. 
Let  $M_z$ be the multiplication operator  $f \mapsto zf$ 
on $S^2 (\mathbb{D})$, whose adjoint 
$M_z^{*} \in \mathcal{L}(S^2 (\mathbb{D})) $ has the form
$$ 
M_z^{*} \big( \sum_{n=0}^{\infty} a_n z^n\big) = a_1 +\sum_{n=1}^{\infty} a_{n+1} \big(\frac{n+1}{n}\big)^n z^n
$$
for $f(z)= \sum_{n=0}^{\infty} a_n z^n \in S^2 (\mathbb{D})$.

\begin{cor}\label{corexuniv} 
Let $\varphi_r$ be as in Theorem \ref{thmadj}. Then the operator 
$$
C_{\varphi_r} - \frac{r}{1+r^2}\big( M_z^{*} + M_z\big) C_{\varphi_r} - 
\frac{1-r^2}{1+r^2}\lambda I 
$$
is universal on $S^2 (\mathbb{D})$ for all $\lambda \in \mathcal{A}_r$.
\end{cor}

\begin{proof}
By \cite[Thm. 6.5]{Heller}, we can write
$$(C_{\varphi_r}^{-1})^{*} -\lambda I = \frac{1+r^2}{1-r^2} C_{\varphi_r} - \frac{r}{1-r^2}\big( M_z^{*} + M_z\big) C_{\varphi_r} -\lambda I +K,$$
where $K $ is a compact operator on $S^2 (\mathbb{D})$.  Recall that
$C_{\varphi_r} $ and $C_{\varphi_r}^{-1} = C_{\varphi_{-r}}$ are similar operators
on $S^2 (\mathbb{D})$. From the symmetry of $ \mathcal{A}_r$
we get that Theorem \ref{thmadj} holds if we replace $C_{\varphi_r}^{*} $ by 
$(C_{\varphi_r}^{-1})^{*} $.  Moreover, the proof of Theorem \ref{thmadj}
shows that  $(C_{\varphi_r}^{-1})^{*} -\lambda I$ 
satisfies condition (C$_+$) on $S^2 (\mathbb{D})$ 
for all $\lambda \in \mathcal{A}_r$. Since the class $\mathcal{U}C_+$ is preserved by compact
perturbations we deduce that 
$ \frac{1+r^2}{1-r^2} C_{\varphi_r} - \frac{r}{1-r^2}\big( M_z^{*} + M_z\big) C_{\varphi_r} -\lambda I $ is a universal operator  on $ S^2 (\mathbb{D})$.
\end{proof}

Recently composition operators have also been studied on the Hardy space
and the weighted Bergman spaces of the upper half-plane $\Pi^+ = \{z \in \mathbb{C}: \textnormal{Im}\, z > 0\}$, 
where new phenomena occur (see e.g. \cite{Ma2, EJ, EW}). Recall that the analytic function 
$F: \Pi^+ \to \mathbb{C}$ belongs to the Hardy space $H^2(\Pi^+)$ if
\[
\Vert F \Vert^2_{H^2(\Pi^+)} = \sup_{y>0} \int_{-\infty}^{\infty} \vert F(x + iy)\vert^2 dx < \infty ,
\]
and to the weighted Bergman space $\mathcal{A}_{\alpha}^2 (\Pi^+)$, for $\alpha >-1$, if
\[
\Vert F \Vert^2_{\mathcal{A}_{\alpha}^2 (\Pi^+)} = \int_0^{\infty} \int_{-\infty}^{\infty} \vert F(x + iy)\vert^2 \, y^{\alpha} \frac{dx \, dy}{\pi} < \infty.
\]
Let $\tau$ be a hyperbolic automorphism of $\Pi^+$, that is,
$\tau(w) = \mu w + w_0$, where $w_0 \in \mathbb{R}$ and $\mu \in (0,1) \cup (1,\infty)$.
It follows from \cite[Thm.  3.1]{EJ} respectively \cite[Thm. 3.4]{EW} that the composition operator 
$C_{\tau}$ is bounded on $H^2(\Pi^+)$ and on $\mathcal{A}_{\alpha}^2 (\Pi^+)$, for all $\alpha >-1$. It is natural to ask whether there is an analogue of the theorem of Nordgren, Rosenthal and Wintrobe on these spaces. 
 
\begin{prop}
The operator $C_{\tau}-\lambda I$ is not universal on $H^2(\Pi^+)$ or
$\mathcal{A}_{\alpha}^2 (\Pi^+)$, where $\alpha >-1$,
 for any $\lambda \in \mathbb{C}$ and any hyperbolic automorphism $\tau$ of $\Pi^+$.
 \end{prop}
 
 \begin{proof}
 The claim  follows from Corollary \ref{bdry} and the spectral results  
\[
\sigma(C_\tau; H^2(\Pi^+)) = \{ \lambda \in \mathbb{C}: \vert \lambda \vert = \mu^{-1/2}\},
\] 
see \cite[Thm. 2.12]{Ma4}, respectively
\[
\sigma(C_\tau; \mathcal{A}_{\alpha}^2 (\Pi^+)) = 
\{ \lambda \in \mathbb{C}:  \vert \lambda \vert = \mu^{-{(\alpha +2)}/2}\}
\] 
for $\alpha >-1$, see  \cite[Thm. 1.2]{S}.  Here $\tau(w) = \mu w + w_0$ and $\mu \in (0,1) \cup (1,\infty)$ as above.
 \end{proof}
 
 \section{Examples of universal commuting pairs}\label{pairs}

Recently M\"uller \cite{Mu2} introduced a notion of universality for commuting pairs
of operators (and more generally, for commuting $n$-tuples). Let $H$ be a separable infinite-dimensional Hilbert space. The commuting pair $(U_1,U_2) \in \mathcal{L}(H)^2$ is said to be  \textit{universal} if for each commuting pair $(S_1,S_2) \in \mathcal{L}(H)^2$ there is a constant  $c \neq 0$ and a subspace $M \subset H$, invariant for both $U_1$ and $U_2$, 
so that  the pairs $({U_1}_{|M},{U_2}_{|M})$ and $(cS_1,cS_2)$ are similar, that is,
there is an isomorphism $V: H \to M$ such that 
$U_1V = cVS_1$ and $U_2V = cVS_2$.

If $(U_1,U_2)$ is a universal commuting pair, then 
$\dim (\textnormal{Ker}\,(U_1) \cap \textnormal{Ker}\,(U_2)) = \infty$, and both  $U_1$ and $U_2$ have to be universal operators for $H$.
M\"uller \cite[Thm. 3]{Mu2} obtained a version of Caradus' condition for the universality of
commuting pairs $(U_1,U_2)$, which we recall next in the special case where $U_1, U_2$
are surjections, see \cite[Cor. 8]{Mu2}.

\begin{itemize}
\item[(M)] \textit{Let $U_1, U_2 \in \mathcal{L}(H)$ be commuting surjections, such that
\begin{itemize}
\item[(i)] $\dim (\textnormal{Ker}\,(U_1) \cap \textnormal{Ker}\,(U_2)) = \infty$, and 
\item[(ii)] $\textnormal{Ker}\,(U_1U_2) = \textnormal{Ker}\,(U_1) + \textnormal{Ker}\,(U_2)$.
\end{itemize}
Then  $(U_1,U_2) \in \mathcal{L}(H)^2$ is a universal commuting pair.
}
\end{itemize}

The following concrete example of a universal commuting pair
is contained in   \cite[Examples 9]{Mu2}: 
Let $H$ be a separable infinite-dimensional Hilbert space and 
$K = \ell^2(\mathbb Z^2_+,H)$, the space of double-indexed sequences with values in $H$.
Define $U_i \in \mathcal{L}(K)$ by $U_if(\alpha) = 
f(\alpha + \beta_i)$ for $\alpha \in \mathbb Z^2_+$
and $f \in  \ell^2(\mathbb Z^2_+,H)$ and $i = 1, 2$, where 
$\beta_1 = (1,0)$ and $\beta_2 = (0,1)$ in $\mathbb Z_+^2$.
Then $(U_1,U_2) \in \mathcal{L}(K)^2$ is a universal commuting pair. (Alternatively,
$U_i = M_{z_{i}}^*$ for $i = 1, 2$, where 
$M_{z_{i}}$ denotes multiplication by the variable $z_i$ in the vector-valued Hardy space
$H^2(\mathbb{D}^2,H)$ for $(z_1, z_2 ) \in \mathbb{D}^2$.)

In this section we are mainly interested in obtaining further concrete examples of 
universal commuting pairs, since it turns out that 
such examples are rather more difficult  
to write down explicitly compared to the class $\mathcal U(H)$. 
Our first observations and examples illustrate 
some of the obstructions, apart from the technical fact that condition (M) requires  
knowledge of  $\textnormal{Ker}\,(U_1)$ and $\textnormal{Ker}\,(U_2)$.
We begin by noting that there is 
a kind of algebraic independence between $U_1$ and $U_2$ for universal pairs 
$(U_1,U_2)$. 
For this let $\{T\}' = \{S \in \mathcal{L}(H): ST = TS\}$ stand for
the commutant of $T$.

\begin{prop}\label{ai}
Let $H$ is a separable infinite-dimensional Hilbert space and $T \in \mathcal{L}(H)$.
\begin{itemize}
\item[(i)] if  $S \in \{T\}'$, then $(T,ST)$ is not a universal commuting 
pair. In particular, 
$(T,p(T)T)$ is not a universal  commuting 
pair for any complex polynomial $p(z) = a_1z  + \ldots + a_nz^n$ satisfying $p(0) = 0$
where $p(T) = a_1T + \ldots + a_nT^n$.
\item[(ii)] $(T^m,T^n)$ is not a universal  commuting 
pair for any $m, n \in \mathbb N$.

\end{itemize}
\end{prop}

\begin{proof}
(i) Consider $(0,I_H) \in \mathcal{L}(H)^2$. If $(T,ST)$ is a universal pair, then corresponding
to the pair $(0,I_H) \in \mathcal{L}(H)^2$ there is an infinite-dimensional subspace $M \subset H$ invariant under $T$, and $c \neq 0$, so that 
$({T}_{|M},{UT}_{|M})$ and $(0,cI_H)$ are similar. However,  the similarity 
of $T_{|M}$ and $0$ implies that $M \subset \textnormal{Ker}\,(T)$, 
so that $UT_{|M}$ cannot be similar to $cI_H$. 

For part (ii) observe that one may assume that $m < n$ by symmetry, whence one may 
argue as in part (i).
\end{proof}

The following example looks at simple ways to 
construct universal pairs starting from 
given universal operators $U, V \in \mathcal U(H)$. 

\begin{example}
(i) Suppose that  $U, V \in \mathcal{L}(H)$ satisfy condition (C), and 
\[
U_0 =  \left( \begin{array}{ccc}
U & 0\\
0 & I_H\\
\end{array} \right), 
\quad V_0 =  \left( \begin{array}{ccc}
I_H & 0\\
0 & V\\
\end{array} \right) \in \mathcal{L}(K),
\]
where $K = H \oplus H$. Then $U_0$ and  $V_0$ are commuting surjections on $K$,
and $\textnormal{Ker}\,(U_0) = \textnormal{Ker}\,(U) \times \{0\}$ and 
$\textnormal{Ker}\,(V_0) = \{0\} \times \textnormal{Ker}\,(V)$. Hence the pair
$(U_0,V_0) \in (\mathcal{L}(K))^2$ is not universal, since  
$\textnormal{Ker}\,(U_0) \cap \textnormal{Ker}\,(V_0) = \{(0,0)\}$.
(Note however that 
$\textnormal{Ker}\,(U_0V_0) = \textnormal{Ker}\,(U_0) +  \textnormal{Ker}\,(V_0)$
in this case.)

(ii) Suppose that $(U_1,U_2)$ and $(V_1,V_2)$ are commuting pairs of surjections
that satisfy condition (M), and put
\[
U =  \left( \begin{array}{ccc}
U_1 & 0\\
0 & V_1\\
\end{array} \right), 
\quad V =  \left( \begin{array}{ccc}
U_2 & 0\\
0 & V_2\\
\end{array} \right) \in \mathcal{L}(K).
\]
In this case $(U,V) \in \mathcal{L}(K)^2$ is a universal commuting pair. 
In fact, $(U,V)$ also satisfies (M), since it is not difficult to check that 
\[
\textnormal{Ker}\,(UV) = \big(\textnormal{Ker}\,(U_1) + \textnormal{Ker}\,(U_2)\big) \times 
\big(\textnormal{Ker}\,(V_1) + \textnormal{Ker}\,(V_2)\big) = 
\textnormal{Ker}\,(U) + \textnormal{Ker}\,(V).
\]
\end{example}

We next look for a non-commutative version of the example from \cite{Mu2}. 
Let $H$ be a separable infinite-dimensional Hilbert space and 
$\mathcal C_2(H)$ the space of Hilbert-Schmidt operators 
on $H$ equipped with the Hilbert-Schmidt norm $\Vert \cdot \Vert_2$. 
Recall that  $T \in \mathcal C_2(H)$ if there is an orthonormal basis 
$(f_n)$ of $H$ such that $\sum_{n=1}^\infty \Vert Tf_n\Vert^2 < \infty$, where
\[
\Vert T \Vert_2 = \Big(\sum_{n=1}^\infty \Vert Tf_n\Vert^2\Big)^{1/2}
\]
 is independent of  the basis.
Then $(\mathcal C_2(H),\Vert \cdot \Vert_2)$ is a separable Hilbert space, and  
$\Vert USV \Vert_2 \le \Vert U\Vert \cdot \Vert V\Vert \cdot \Vert S\Vert_2$ for 
$S \in \mathcal C_2(H)$ and $U,V \in \mathcal{L}(H)$. 
We refer e.g. to  \cite[chapter 2.4]{Mur} for more background on the ideal
$\mathcal C_2(H)$ of $\mathcal{L}(H)$.

Hence the multiplication maps 
$L_U$ and $R_U$ are bounded  operators 
$\mathcal C_2(H) \to \mathcal C_2(H)$, where 
$L_U(S) = US$ and $R_U(S) = SU$ for any $U \in \mathcal{L}(H)$
and  $S \in \mathcal C_2(H)$.
Clearly $(L_U,R_V) \in \mathcal{L}(\mathcal C_2(H))^2$ is a commuting
 pair for any $U, V \in \mathcal{L}(H)$,
and we are interested in the universality of $(L_U,R_V)$ on $\mathcal C_2(H)$.
Let $B$ be the standard backward shift on $\ell^2 = \ell^2(\mathbb{Z}_+)$, that is,
$B(x_0,x_1, \ldots) = (x_1, x_2, \ldots)$ for $(x_j) \in \ell^2$. 
It turns out that $(L_B,R_{B^*})$ is not a universal pair 
(see part (ii) of Theorem \ref{unipair}), and to obtain universal pairs
$(L_U,R_V)$ we will consider the vector-valued direct $\ell^2$-sum $\mathcal H = 
\ell^2(\mathbb Z_+,H) = (\oplus_{n \in \mathbb Z_{+}} H)_{\ell^2}$, where 
$H$ is  a fixed separable infinite-dimensional Hilbert space. 
Let $B_\infty$ be the 
backward shift of infinite multiplicity on $\mathcal H$, 
so that $B^*_\infty$ is the corresponding forward shift on $\mathcal H$. 

The following result contains the main example of this section. We will use  $u \otimes v$ 
for given $u, v \in H$ to denote the rank-$1$ operator  
$x \mapsto \langle x,u\rangle v$ on $H$.

\begin{theorem}\label{unipair}
\begin{itemize}
\item[(i)]   $L_B, R_{B^*} \in \mathcal{U} (\mathcal C_2(\ell^2))$ and
 $L_{B_{\infty}}, R_{B^*_{\infty}} \in \mathcal{U} (\mathcal C_2(\mathcal H))$.\\
\item[(ii)] $(L_{B},R_{B^*})$ is not a universal pair on $\mathcal C_2(\ell^2)$.\\
\item[(iii)] $(L_{B_{\infty}},R_{B^*_{\infty}})$ is a universal pair 
on $\mathcal C_2(\mathcal H)$.
\end{itemize}
\end{theorem}

\begin{proof}
(i) We check that   $L_B$ and $R_{B^*}$ satisfy condition (C) on $\mathcal C_2(\ell^2)$
In fact, if $(e_n)$ is the standard unit vector basis of $\ell^2$ 
then $L_B(e_n \otimes e_0) = e_n \otimes Be_0 = 0$ for any $n \in \mathbb Z_+$, so that 
$\textnormal{Ker}\,(L_B)$ is infinite-dimensional. 
Moreover, $\textnormal{Ran}\,(L_B) = \mathcal C_2(\ell^2)$
since $BB^* = I_H$, and the argument for $R_{B^*}$ is similar. 

The universality of  $L_{B_{\infty}}$ and  $R_{B^*_{\infty}}$ on $\mathcal C_2(\mathcal H))$
follows from part (iii) 
(alternatively, one may also modify the preceding argument as in the proof of (iii)).

(ii) Recall that
$(e_n \otimes e_m)_{m, n \in \mathbb Z_{+}}$ is an orthonormal basis of  $\mathcal C_2(\ell^2)$. 
It follows from the identities $L_B(e_n \otimes e_m) = e_n \otimes Be_m$ and  
$R_{B^*}(e_n \otimes e_m) = Be_n \otimes e_m$ that
 $\textnormal{Ker}\,(L_B) = [e_n \otimes e_0: n \in \mathbb Z_+]$ and 
  $\textnormal{Ker}\,(R_{B^*}) = [e_0 \otimes e_n: n \in \mathbb Z_+]$. 
  Here $[A]$ denotes the closed linear
  span of the set $A \subset \mathcal C_2(\ell^2)$.  In particular,
 $\textnormal{Ker}\,(L_B) \cap \textnormal{Ker}\,(R_{B^*}) = 
[e_0 \otimes e_0]$
 is $1$-dimensional, so that $(L_{B},R_{B^*})$ can not be a universal pair.
  
(iii) 
We will verify that condition (M) holds.
 Towards this note first that
 $L_{B_{\infty}}$ and $R_{B^*_{\infty}}$ are surjections on 
$\mathcal C_2(\mathcal H)$, since  $B_{\infty}B^*_{\infty} = I_{\mathcal H}$
implies  that  $L_{B_{\infty}}(B^*_{\infty}S) = S$ and
$R_{B^*_{\infty}}(SB_\infty) = S$ for any $S \in \mathcal C_2(\mathcal H)$.

To compute the kernels of $L_{B_{\infty}}$ and $R_{B^*_{\infty}}$
we need the fact that any $S \in \mathcal C_2(\mathcal H)$ is uniquely determined by its 
operator-matrix components $S_{i,j} = P_iSJ_j \in \mathcal C_2(H)$ for $i,j \in \mathbb Z_+$.
Here $P_i$ is the orthogonal projection $\mathcal H \to H$ onto the $i$:th copy of $H$,
 and $J_j: H \to \mathcal H$ the canonical inclusion from the $j$-th copy. 
One deduces from the definition of $S_{i,j}$ and the identity 
$L_{B_{\infty}}(u \otimes v) = u \otimes B_{\infty}v$ for $u, v \in \mathcal H$
that 
\[
L_{B_{\infty}}(S) = 
\begin{pmatrix}
     S_{1,0} 	& S_{1,1} & \cdots  \\
    S_{2,0} & S_{2,1} 	& \cdots  \\
  \vdots  	& \vdots 				& \ddots  
\end{pmatrix}
\]
for $S = (S_{i,j})$, so that by uniqueness 
\begin{equation}\label{leftker}
\textnormal{Ker}\,(L_{B_{\infty}}) = \{S = (S_{i,j}) \in \mathcal C_2(\mathcal H): S_{i,j} = 0 \textrm{ for } i > 0\}.
\end{equation}
Similarly, the identity 
$R_{B^*_{\infty}}(u \otimes v) = B_\infty u \otimes v$ for $u, v \in \mathcal H$
yields that 
\[
R_{B^*_{\infty}}(S) = 
\begin{pmatrix}
     S_{0,1} 	& S_{0,2} & \cdots  \\
    S_{1,1} & S_{1,2} 	& \cdots  \\
  \vdots  	& \vdots 				& \ddots  
\end{pmatrix}, 
\]
whence 
\begin{equation}\label{rightker}
\textnormal{Ker}\,(R_{B^*_{\infty}}) = \{S = (S_{i,j}) \in \mathcal C_2(\mathcal H): S_{i,j} = 0 \textrm{ for } j > 0\}.
\end{equation}
In particular, we get  that 
\[
\textnormal{Ker}\,(L_{B_{\infty}}) \cap \textnormal{Ker}\,(R_{B^*_{\infty}}) = \{S = (S_{i,j}) \in \mathcal C_2(\mathcal H):
S_{i,j} =0 \textrm{ for } i > 0 \textrm{ or }  j > 0\}
\]
is infinite-dimensional, since the operator $S_{0,0} \in \mathcal C_2(H)$ can be chosen freely. 
Finally, we need to verify  that 
\begin{equation}\label{upair}
\textnormal{Ker}\,(L_{B_{\infty}}R_{B^*_{\infty}}) = \textnormal{Ker}\,(L_{B_{\infty}}) + \textnormal{Ker}\,(R_{B^*_{\infty}}).
\end{equation}
However, \eqref{upair} follows  from \eqref{leftker} and \eqref{rightker},
the identity 
$L_{B_{\infty}}R_{B^*_{\infty}}(u \otimes v) = B_{\infty}u \otimes B_{\infty}v$ 
for $u, v \in \mathcal H$,   the observation that 
\[
L_{B_{\infty}}R_{B^*_{\infty}}(S) = 
\begin{pmatrix}
     S_{1,1} 	& S_{1,2} & \cdots  \\
    S_{2,1} & S_{2,2} 	& \cdots  \\
  \vdots  	& \vdots 				& \ddots  
\end{pmatrix} 
\]
as well as the fact that 
\[
\begin{pmatrix}
     S_{0,0} 	& S_{0,1} & \cdots  \\
    S_{1,0} & 0 	& \cdots  \\
  \vdots  	& \vdots 				& \ddots  
\end{pmatrix}
= \begin{pmatrix}
     S_{0,0} 	& S_{0,1} & \cdots  \\
   0 & 0 	& \cdots  \\
  \vdots  	& \vdots 				& \ddots  
\end{pmatrix}
+ 
\begin{pmatrix}
     0 	& 0 & \cdots  \\
    S_{1,0} & 0 	& \cdots  \\
  \vdots  	& \vdots 				& \ddots  
\end{pmatrix}
\]
is the sum of two well-defined Hilbert-Schmidt operators for any 
$S = (S_{i,j}) \in \mathcal C_2(\mathcal H)$, see e.g. the proof of \cite[1.c.8]{LT1}.

We conclude from condition (M) that  $(L_{B_{\infty}},R_{B^*_{\infty}})$
 is a universal  pair.
\end{proof}

\noindent \textit{Remarks.} By a straightforward modification of the argument in 
part (iii) one may also 
show that $(L_{(B_{\infty})^m},R_{(B^*_{\infty})^n})$
 is a universal  pair on $\mathcal C_2(\mathcal H)$ for any $m,n \in \mathbb N$. 
We do not know explicit conditions on $(U,V)$ 
which ensure that $(L_U,R_V)$ is a universal pair on $\mathcal C_2(H)$. 

\medskip

For our last examples we return to the setting (and notations) from section \ref{secunivex}
related to composition operators on $H^2 (\mathbb{D})$
associated to hyperbolic automorphisms of $\mathbb{D}$. Recall that 
$C_{\varphi_r}-\lambda  I$ and  $C_{\varphi_s}-\mu I$ commute for any $0 < r, s < 1$ and 
$\lambda, \mu \in \mathbb{C}$. This follows from the fact that
\begin{equation}\label{sgrp}
\varphi_r \circ \varphi_s = \varphi_t = \varphi_s \circ \varphi_r
\end{equation}
where $t = \frac{r+s}{1+rs}$. The result of Nordgren, Rosenthal and Wintrobe \cite{NRW}
suggests the following natural question.

\medskip

\noindent \textit{Problem.}
Are there universal pairs of the form
\[
\big(C_{\varphi_r}-\lambda  I , C_{\varphi_s}-\mu I \big) \in \mathcal{L}(H^2 (\mathbb{D}))^2,
\]
for some $0 < r, s < 1$, $\lambda  \in \mathcal{A}_r$ and $\mu \in \mathcal{A}_s$  ?

\medskip

We first note some obvious restrictions in view of Proposition \ref{ai}.

\begin{example}
(i) The pair $(C_{\varphi_{r}} - \lambda I, C_{\varphi_{r}} - \mu I)$ is  not 
universal for any  $0 < r < 1$ and $\lambda \neq \mu$. In fact, in this case 
$\textnormal{Ker}\,(C_{\varphi_{r}} - \lambda I) \cap \textnormal{Ker}\,(C_{\varphi_{r}} - \mu I) 
= \{0\}$ once $\lambda \neq \mu$.

(ii) Let $0 < r < 1$. By \eqref{sgrp} there is $r_n \in (0,1)$ 
such that $\varphi ^n_{r} = \varphi_{r_{n}}$, 
where $\varphi ^n_{r} = \varphi_{r} \circ \ldots \circ \varphi_{r}$ ($n$-fold composition). Then
$(C_{\varphi_{r}} - \lambda I, C_{\varphi_{r_{n}}} - \lambda^n I)$ 
is not a universal pair for any  $n \ge 2$ and $\lambda \in \mathcal{A}_r$.
Indeed, to see this we write $C_{\varphi_{r_{n}}} - \lambda^n I = S(C_{\varphi_{r}} - \lambda I)$,
where $S$ commutes with $C_{\varphi_{r}}$, and apply  Proposition \ref{ai}.
\end{example}

In our final example we use the recent approach of Cowen and Gallardo-Guti\'{e}rrez 
\cite{CG2} to the universality result by Nordgren, Rosenthal and Wintrobe
in order to analyse more carefully a pair, which shows 
subtler obstructions related to the existence of NRW-pairs.

\begin{example}\label{NRWex}
There are $r, s \in (0,1)$ and respective eigenvalues  $\lambda \in \mathcal{A}_r$,
$\mu \in \mathcal{A}_s$ such that 
$\textnormal{Ker}\,(C_{\varphi_{r}} - \lambda I) \cap \textnormal{Ker}\,(C_{\varphi_{s}} - \mu I)$
is infinite-dimensional, but condition (M) fails to hold for the pair
$\big(C_{\varphi_r}-\lambda I , C_{\varphi_s}-\mu I \big)$.
\end{example}

\begin{proof}
Fix $0 < r < 1$ and consider the positive 
eigenvalue $\lambda = (\frac{1-r}{1+r})^{-u} \in \mathcal{A}_r$, where $0 < u < 1/2$. 
It follows from the proofs of Lemma 7.3 and Theorem 7.4 in \cite{CM}
that the functions
\[
f_n(z) = \exp \Big( (u+ i n2\pi a) \log \frac{1+z}{1 - z}\Big),
\]
where $a = (\log \frac{1-r}{1+r})^{-1}$ and $n \in \mathbb{Z}$,
forms a linearly independent family of eigenfunctions for $C_{\varphi_{r}}$ in $H^2 (\mathbb{D})$
associated with the eigenvalue $\lambda$. Here the logarithm refers to the principal branch.
Let $r < s < 1$ and 
consider  $\mu = (\frac{1-s}{1+s})^{-u} \in \mathcal{A}_s$. As above 
\[
g_m(z) = \exp \Big( (u+ i m 2\pi b) \log \frac{1+z}{1 - z}\Big),
\]
where $b = (\log \frac{1-s}{1+s})^{-1}$ and $m \in \mathbb{Z}$,
are linearly independent 
eigenfunctions of $C_{\varphi_{s}}$ for the eigenvalue $\mu$. 
Moreover, we select  $r$ and $s$  so that  
 \begin{equation}\label{infsol}
 \frac{n}{m} = \frac{b}{a} = \frac{\log \frac{1+r}{1-r}}{\log \frac{1+s}{1-s}}
 \end{equation}
holds for  infinitely many pairs
$n, m \in \mathbb{Z} \setminus \{0\}$. 
 This ensures that 
 $\textnormal{Ker}\,(C_{\varphi_{r}} - \lambda I) \cap \textnormal{Ker}\,(C_{\varphi_{s}} - \mu I)$
is infinite-dimensional by comparing the above eigenfunctions. 

Recall from \cite[Thm. 5]{CG2011} (see also \cite{CG2}) that there is 
an analytic covering map $\psi_r: \mathbb{D} \to \mathcal{A}_r$, such that 
$C_{\varphi_{r}} - \mu I$ and $T^*_{\psi_{r}- \lambda}$ are similar operators 
on $H^2 (\mathbb{D})$, where $T_{\psi_{r}- \lambda} \in \mathcal{L}(H^2 (\mathbb{D}))$
is the analytic Toeplitz operator
$f \mapsto (\psi_{r}- \lambda)f$. Indeed,
$$
\psi_r (z) = \Big(\frac{1-z}{1+z} \Big)^{i t_r/\pi},
$$
where $ t_r = -\log \big( \frac{1-r}{1+r}\big)>0$. Moreover, 
$C_{\varphi_{s}} - \mu I$ and $T^*_{\psi_{s}- \mu}$ are similar operators
on $H^2 (\mathbb{D})$
for the covering map $\psi_s: \mathbb{D} \to \mathcal{A}_s$, where $ \psi_s (z) = \Big(\frac{1-z}{1+z} \Big)^{i t_s/\pi}$ and $ t_s = -\log \big( \frac{1-s}{1+s}\big)$.

By standard duality, the dual version of part (ii) of condition (M) for the pair 
$(T^*_{\psi_{r}- \lambda},T^*_{\psi_{s}- \mu})$ is the requirement that
\begin{equation}\label{ran}
\textnormal{Ran}\,(T_{\psi_{s} - \mu}T_{\psi_{r}- \lambda}) 
= \textnormal{Ran}\,(T_{\psi_{r}- \lambda})  \cap \textnormal{Ran}\,(T_{\psi_{s}- \mu}).
\end{equation}
Note for this that all the ranges are closed, since the adjoints are onto maps by similarity. 

We claim that condition \eqref{ran} does not hold. Towards this  
consider the standard factorisation 
$\psi_{r}- \lambda = B_1S_1F_1$
into a Blaschke product $B_1$ containing the zeroes of the function (counting 
multiplicity), a singular inner function $S_1$ and an outer function $F_2$. 
Let $\psi_{s}- \mu = B_2S_2F_2$ be the analogous factorisation for $\psi_{s}$.
It is  not difficult to check from \eqref{infsol} and the explicit form of 
$\psi_{r}$ and $\psi_{s}$ that 
the functions $\psi_{r}- \lambda$ and $\psi_{s}- \mu$ have infinitely many simple common zeroes.
Let $B_1 = B_0B_3$ and $B_2 = B_0B_4$, where $B_0$ is the Blaschke product 
which contains the common zeroes.

To conclude the argument consider the function $g = B_0B_3B_4S_1F_1S_2F_2$. Observe that
$g \in H^2 (\mathbb{D})$ since the maps $\psi_{r}- \lambda$ and $\psi_{s}- \mu$, and hence also $S_1F_1$ and $S_2F_2$,
belong to $H^\infty$ (see \cite[Thm. 17.9]{Ru}, for instance).
Clearly 
$g \in   \textnormal{Ran}\,(T_{\psi_{r}- \lambda})  \cap \textnormal{Ran}\,(T_{\psi_{r}- \lambda})$
by inspection.
However, the Blaschke product containing the zeroes of 
$(\psi_{r}- \lambda)(\psi_{s}- \mu)$ 
already has the form $B_0^2B_3B_4$, so that 
$g \notin  \textnormal{Ran}\,(T_{\psi_{s} - \mu}T_{\psi_{r}- \lambda})$
in view of the uniqueness of the factorisation.
\end{proof}

\noindent \textit{Remarks.} One may verify from  \eqref{infsol} that 
the pair $\big(C_{\varphi_r}-\lambda I , C_{\varphi_s}-\mu I  \big)$ 
studied in Example \ref{NRWex}
has the property that 
there is $p, q \in \mathbb{N}$ with $p < q$, for which 
$C^q_{\varphi_r}-\lambda^q I = C^p_{\varphi_s}-\mu^p I$. However, we do not have
general results which would exclude such a property for universal pairs.

\section*{Acknowledgements}

The authors thank Isabelle Chalendar for drawing the attention to reference \cite{Mu2}, and Pekka Nieminen for useful comments.
The first author is grateful to  Eva Gallardo-Guti\'{e}rrez for several illuminating discussions on  universal operators during the past years. 

This paper is part of the first author's PhD thesis, which is 
supervised by the second author and Pekka Nieminen. 

\bibliographystyle{plain}

\end{document}